\documentclass{amsart}
\usepackage{amsmath,amsfonts,amssymb,amsthm,epsfig,enumitem,listings,mathrsfs,setspace}
\usepackage{color, bm}
\usepackage{comment}
\usepackage[colorlinks=true, pdfstartview=FitV, linkcolor=blue, citecolor=blue, urlcolor=blue, bookmarksdepth=2]{hyperref}
\usepackage[all]{hypcap}
\usepackage[vcentermath]{youngtab}
\usepackage{tikz}
\usetikzlibrary{matrix}
\tikzset{tab/.style={matrix of math nodes,column sep=-.4, row sep=-.4,text height=8pt,text width=8pt,align=center,inner sep=3}}
\tikzset{smalltab/.style={matrix of math nodes,column sep=-.4, row sep=-.4,text height=6pt,text width=6pt,align=center,inner sep=2.5}}

\usepackage{appendix}

\numberwithin{equation}{section}
\numberwithin{figure}{section}
\numberwithin{table}{section}
\newtheorem{Theorem}[equation]{Theorem}

\newtheorem{Lemma}[equation]{Lemma}

\newtheorem{Corollary}[equation]{Corollary}

\theoremstyle{definition}

\newtheorem{Definition}[equation]{Definition}
\newtheorem{Example}[equation]{Example}
\newtheorem{Remark}[equation]{Remark}

\newcommand{\arxiv}[1]{\href{http://arxiv.org/abs/#1}{\tt arXiv:\nolinkurl{#1}}}

\newcommand{\wt}{{\rm wt}}

\newcommand{\Irr}{{\bm Irr}}
\newcommand{\bfv}{{\bm v}}
\newcommand{\Hom}{\text{Hom}}
\newcommand{\g}{\mathfrak{g}}
\newcommand{\im}{\operatorname{im}}

\newcommand{\End}{\text{End}}

\newcommand{\so}{\rm so}
\newcommand{\ta}{\rm ta}

\begin{document}

\title[A quiver variety approach to root multiplicities]{Quiver varieties and root multiplicities in rank 3}

\author{Patrick Chan}
\address{Department of Mathematics and Statistics, Loyola University, Chicago, IL}
\email{pchan2@luc.edu}

\author{Peter Tingley}
\address{Department of Mathematics and Statistics, Loyola University, Chicago, IL}
\email{ptingley@luc.edu}

\begin{abstract}
Building on our previous work in rank two, we use quiver varieties to give a combinatorial upper bound on dimensions of certain imaginary root spaces for rank 3 symmetric Kac-Moody algebras. We describe an explicit method for extracting combinatorics when the Dynkin diagram is bipartite (i.e. two of the nodes are not connected). As in rank two we believe these bounds are quite tight and we give computational evidence to this effect, although there is more error in rank 3 than in rank 2. 
\end{abstract}
 
\maketitle

\section{Introduction}
In \cite{T21}, we developed a method for studying root multiplicities of symmetric Kac-Moody algebras using quiver varieties. This gives a framework for finding combinatorial upper bounds on the multiplicities and we conjectured that the resulting bounds, at least in some cases, are quite tight. While the construction is general, we only translated it to combinatorics in rank two. There the upper bound consisted of the number of rational Dyck paths satisfying various conditions. 

Here we develop explicit combinatorics using the same method in certain rank three cases. Specifically, we consider a Kac-Moody algebra $\mathfrak{g}$ with Dynkin diagram
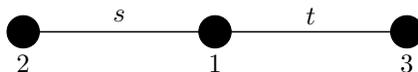
\begin{center}
\begin{tikzpicture}[scale=0.85]
\node [draw, circle, fill=black] at (0,0) {.};
\node [draw, circle, fill=black] at (3,0) {.};
\node [draw, circle, fill=black] at (6,0) {.};
\draw[line width = 0.02cm] (0,0)--(3,0);
\draw[line width = 0.02cm] (3,0)--(6,0);

\draw node at (0,-0.5){2};
\draw node at (3,-0.5){1};
\draw node at (6,-0.5){3};
\draw node at (1.5,0.25) {$s$};
\draw node at (4.5,0.25) {$t$};
\end{tikzpicture}
\end{center}
\noindent meaning there are $s$ edges on the left and $t$ edges on the right. 
Our main result is the following:
\begin{Theorem} \label{th:main} 
For any imaginary root $\beta= a\alpha_1+b\alpha_2+c \alpha_3$ with $\text{gcd}(a,b,c)=1$, the root multiplicity $m_\beta$ is bounded above by the number of words of the form 
\begin{equation}1^{a_1} 2^{b_1}3^{c_1} 1^{a_2} 2^{b_2} 3^{c_2} \cdots,\end{equation}
where the number of $1$s is a, the number of 2s is b, the number of 3s is c, 
and, 
\begin{enumerate}

\item \label{main0} For each $i$, $a_i \neq 0$ and $b_i$ or $c_i$ is also non-zero.

\item \label{main1} Draw a path in the plane by drawing each 1 as a vertical line of length 1, each 2 as a horizontal line of length $s$, and each 3 as a horizontal line of length $t$. The result is a rational Dyck path. That is, a straight line connecting the beginning and end of the path stays weakly below the path. 

\item \label{main2}  If a prefix $1^{a_1} 2^{b_1}3^{c_1} \cdots 1^{a_k} 2^{b_k}3^{c_k}$ corresponds to a point where the Dyck path touches the diagonal, then $\displaystyle  \frac{b_1+\cdots +b_k}{c_1+\cdots+c_k} \geq  \frac{b}{c}.$ Equivalently, if you deform the path to have each 2 correspond to an edge of length $s-\epsilon$, the resulting path still stays above its diagonal. 

\item \label{part:main3easy-s} For each $i \geq 1$,   $\displaystyle  \frac{b_i}{a_i} \leq s$ and  $\displaystyle  \frac{c_i}{a_i} \leq t$.

\noindent $\mbox{}$ \hspace{-1.7cm} For each $i\geq 1,$ let $n_{Bi}=  \min \{ b_i, s a_{i+1}-b_{i+1}\}$,  $n_{Ci}=  \min \{ c_i, t a_{i+1}-c_{i+1}\}.$ Then

\item  \label{part:combo} $\displaystyle a_{i+1} \leq s n_{Bi}+tn_{Ci}- \max \{s^{-1} n_{Bi}, t^{-1} n_{Ci} \}.$

\item \label{part:ratio} $\displaystyle \frac{a_{i+1}}{sn_{Bi}+tn_{Ci}} \leq \frac{1}{2}+\frac{\sqrt{(s^2+t^2)^2-4(s^2+t^2)}}{2(s^2+t^2)}$. Note: if $sn_{Bi}+tn_{Ci}=0$ then condition \eqref{part:combo} would have been violated.

\end{enumerate}
\end{Theorem}

Similar methods should work for any symmetric Kac-Moody algebra whose Dynkin diagram is bipartite, with the combinatorics being most similar when the Dynkin diagram is a star. The method from \cite{T21} is even more general, but as the diagram gets more complex it is harder to describe things combinatorially. 

Theorem \ref{th:main} does not give the complete list of ``local" conditions on the path, so the upper bound will diverge exponentially from the correct answer. This is in contrast to the rank 2 situation, where we believe we have listed all local conditions. Nonetheless we give evidence that the bounds are meaningful. We also discuss ways to tighten the bounds.

Root multiplicities have been studied quite extensively, see \cite{CFL:2014} for a survey. They can be calculated exactly, see \cite{BM,Pete:rec}, although the formulae are complicated.
Special cases have been further investigated and combinatorialized in \cite{FF, KLL, KM95}. But open questions remain. See e.g. \cite[Open Problems 2 and 3]{CFL:2014} and Frenkel's conjectural upper bound for hyperbolic cases \cite{Frenkel-conj}. We hope our work sheds some light on the situation, particularly on asymptotics. 

\section{Background}

\subsection{Kac-Moody algebras and $B(-\infty)$}
Consider the Kac-Moody algebra $\mathfrak{g}$ associated to a symmetric Cartan matrix $A=(a_{ij})$ with index set $I$, as in \cite{Kac:1990}. The Lie algebra $\g$ is graded by the root lattice $Q$, which is the ${\Bbb Z}$-span of the simple roots $\alpha_i$ for $i \in I$. By definition a non-zero $\beta \in Q$ is a root if $\dim \g_\beta \neq 0$, in which case $m_\beta:= \dim \g_\beta $ is called the root multiplicity. All roots are either positive, meaning they are ${\Bbb Z}_{\geq 0}$ linear combinations of the $\alpha_i$, or negative, meaning the negatives of these. Let $\Delta$ denote the set of roots and $\Delta_+$ the set of positive roots.

There is an inner product on $Q$ defined by, for simple roots $\alpha_i,\alpha_j$, $\langle \alpha_i, \alpha_j \rangle = a_{ij}$. All roots $\beta$ have the property that either $\langle \beta, \beta \rangle=2$, in which case $\beta$ is called a real root, or $\langle \beta, \beta \rangle \leq 0$, in which case $\beta$ is called an imaginary root.
The reflection corresponding to each $\alpha_i$ acts on weight space by $$s_i(\nu)= \nu- \langle \alpha_i , \nu \rangle \alpha_i.$$
This has various important properties, including:
\begin{itemize}
\item It preserves the inner product.

\item It preserves root multiplicities.

\item It preserves the set of positive imaginary roots, meaning it acts as a permutation on this set. 
\end{itemize}

\subsection{The crystal $B(-\infty)$}
We don't work directly with $\g$ here but instead with the crystal $B(-\infty)$. This is a combinatorial object related most closely not to $\g$ but to its universal enveloping algebra $U(\mathfrak{g})$. As a vector space 
$$U(\g)= U^-(\g) \otimes U^0(\g) \otimes U^+(\g),$$
where $U^-(\g), U^0(\g), U^+(\g)$ are the subalgebras generated by the negative root spaces, the Cartan subalgebra, and the positive root spaces, respectively. The graded dimension of $U^+(\g)$ is 
$$\dim U^+(\g) = \prod_{\beta \in \Delta_+} \left( \frac{1}{1-e^\beta} \right)^{m_\beta}.$$
That is, the dimension of $U^+(\g)_\gamma$ is the number of ways to write $\gamma$ as a sum of positive roots, taking into account multiplicities. This is also called the Kostant partition function of $\gamma$. 

The crystal $B(-\infty)$ is a set along with operators $e_{i}, f_{i} \colon B(-\infty) \rightarrow B(-\infty) \cup \{ \emptyset \}$ for each $i \in I$ which satisfy various axioms. There is a weight function $\wt: B(-\infty) \rightarrow Q$ and the number of elements of a given weight $\gamma$ is the dimension of the $\gamma$ weight space in $U^+(\mathfrak{g})$.
See \cite{Kashiwara:1995} or \cite{HK:2002} (which consider $B(\infty)$, the Cartan involution of $B(-\infty)$).  Here we use the geometric realization of $B(-\infty)$ from \cite{KS:1997}, explained below.

\subsection{Quiver varieties } \label{ss:QV}

Fix a graph $G$ with vertex set $I$ and edge set $E$. Let $A$ be the set of arrows, so there are two arrows for each edge $e \in E$, one pointing in each direction. For each arrow $a$ let $\so(a)$ be the source and $\ta(a)$ be the target, meaning $a$ points from $\so(a)$ to $\ta(a)$. 
\begin{Definition}
The path algebra ${\Bbb C}[G]$ is the ${\Bbb C}$-algebra with basis consisting of all paths in $G$ (sequences of arrows $a_k \cdots a_1$ with $\so(a_{i+1})=\ta(a_i)$ along with the lazy paths $e_i$ at each vertex) and multiplication given by 
$$(b_{k} \cdots b_1) (a_j \cdots a_1) = 
\begin{cases}
b_k \cdots b_1 a_j \cdots a_1 \quad \quad \text{ if } \so(b_1)=\ta(a_j) \\
0 \quad \quad \quad \quad \quad  \quad \quad \quad \ \   \text{otherwise}.
\end{cases}
$$
\end{Definition}
\noindent Choose a subset $\Omega$ of $A$ where each edge appears in exactly one direction (sometimes called an orientation of the graph) and define $\epsilon(a)= 1$ if $a \in \Omega$ and $-1$ otherwise. For $a \in  A$, let $\bar a$ denote the reverse arrow. 

\begin{Definition}
The preprojective algebra $\Lambda$ is the quotient of ${\Bbb C}[G]$ by the ideal generated by
$$\epsilon =   \sum_{a \in A} \epsilon (a) \bar a a.$$
\end{Definition}

\begin{Definition}
For any $I$-graded vector space $V = \oplus_I V_i$, let $\Lambda(V)$ be the variety of actions of $\Lambda$ on $V$ where the lazy path at $i$ acts as projection onto $V_i$, and which are nilpotent in the sense that all paths of length at least $\dim V$ act as $0$. 
\end{Definition}

An action of ${\Bbb C}[G]$ on $V$ is determined by a homomorphism for each arrow so can be described as $$(x_a)_{a\in A} \in \oplus_A \Hom (V_{\so(a)}, V_{\ta(a)}).$$
$\Lambda(V)$ is a sub-variety of this space. Up to isomorphism it depends only on the dimension vector $\bfv =(v_i)_{i \in I}$, where $v_i= \dim V_i$. 

Associate to $G$ a symmetric Cartan matrix whose index set is the set of vertices, the diagonal entries are $2$, and other entries are defined by setting $-a_{ij}$ to be the number of edges connecting $i$ and $j$. We identify a dimension vector $\bfv$ with $\gamma= \sum_i v_i \alpha_i \in Q$ and sometimes denote $\Lambda(V)$ by $\Lambda(\gamma)$. Denote the set of irreducible components of $\Lambda(\gamma)$ by $\Irr \Lambda (\gamma)$.
The following is due to Kashiwara and Saito \cite{KS:1997}, see also \cite[Proposition 3.14]{NT}.

\begin{Theorem} \label{th:QVrealization}
The crystal $B(-\infty)$ can be indexed by 
$\coprod \Irr \Lambda (\gamma)$. The operation $f_i^{\text{max}}$ which applies the crystal operator $f_i$ as many times as possible before returning $\emptyset$ acts on $X \in \Irr \Lambda (\bfv)$ as follows:
Fix $T \in X$. Let $\text{Soc}_i(T)$ be the intersection of the socle of $T$  with $V_i$ and set $\gamma'= \gamma-\dim \text{Soc}_i(T) \alpha_i$. Generically $T/{\text{Soc}_i}(T)$ is isomorphic to a point in a unique $Y \in \Irr \Lambda(\gamma')$, and $f_i^\text{max} X =Y$.   \qed
\end{Theorem}

\begin{Example}
The case $s=2,t=1$ is particularly important here. That is, the Dynkin diagram 
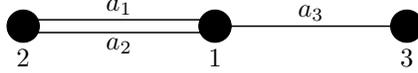
\begin{center}
\begin{tikzpicture}[scale=0.85]
\node [draw, circle, fill=black] at (0,0) {.};
\node [draw, circle, fill=black] at (3,0) {.};
\node [draw, circle, fill=black] at (6,0) {.};
\draw[line width = 0.02cm] (0,-0.1)--(3,-0.1);
\draw[line width = 0.02cm] (0,0.1)--(3,0.1);
\draw[line width = 0.02cm] (3,0)--(6,0);

\draw node at (0,-0.5){2};
\draw node at (3,-0.5){1};
\draw node at (6,-0.5){3};
\draw node at (1.5,0.3) {$a_1$};
\draw node at (1.5,-0.3) {$a_2$};
\draw node at (4.5,0.2) {$a_3$};
\end{tikzpicture}
\end{center}

\noindent with two edges on the left and one on the right. Orient all the edges to point away from the center node, and call  the oriented arrows $a_1,a_2,a_3$. 
A representation of ${\Bbb C}[G]$ on $V_1 \oplus V_2 \oplus V_3$ consists of 6 maps:
$$x_{a_1}, x_{a_2} : V_1 \rightarrow V_2, \quad x_{a_3}: V_1 \rightarrow V_3, \quad x_{\bar a_1}, x_{\bar a_2}: V_2 \rightarrow V_1, \quad x_{\bar a_3}: V_3 \rightarrow V_1.$$ 
$\Lambda(v)$ is the sub-variety cut out by the condition that all paths of length greater than $\dim V_1+\dim V_2+\dim V_3$ act as 0  and the preprojective relations
$$
\bar a_1 a_1+\bar a_2 a_2 +\bar a_3 a_3=0, \quad a_1 \bar a_1+ a_2 \bar a_2=0, \quad  a_3 \bar a_3=0 ,
$$
where these are equations in $\End (V_1), \End (V_2), \End(V_3)$ respectively. 
\end{Example}

\subsection{Stability conditions}
We loosely follow \cite{BKT:2014}, drawing on notation from \cite{TW} as well. See also \cite{T21}.

\begin{Definition}
A {\bf charge} $c$ is a linear function $c: Q \rightarrow {\Bbb C}$ such that $c(\alpha_i)$ is in the upper half plane for all simple roots $\alpha_i$. 
\end{Definition}
For a fixed charge $c$, each representation $V$ of $\Lambda$ has a unique filtration 
$$\emptyset = V_0 \subset V_1 \subset \cdots \subset V_k= V$$
where the sub-quotients $\overline V_i=V_i/V_{i-1}$ satisfy
\begin{enumerate}
\item[(HN1)] $ \overline V_i$ has no submodule $S$ with $\arg(c(\dim S))< \arg(c( \dim  \overline V_i))$,

\item [(HN2)] $\arg(c (\dim  \overline V_1)) < \arg(c (\dim  \overline V_2)) < \cdots < \arg(c (\dim  \overline V_k))$.

\end{enumerate}
Here $\arg(z)$ is the angle formed by $z,0$ and $1$ in the complex plane. We call this the HN filtration since it is a special case of a Harder-Narasimhan filtration as in e.g. \cite{Rudakov97}.

\subsection{String data} \label{ss:stringdata}
The following was first studied by Kashiwara 
\cite[\S 8.2]{Kashiwara:1995} and  
was further developed by Littelmann \cite{Littelmann:1998}. 
Choose a sequence 
$i_1, i_2, i_3 \ldots$
of nodes in the Dynkin diagram with each appearing infinitely many times. 
The {\bf string data} $(a_1,a_2, \ldots)$ of $b \in B(-\infty)$ is 
\[
\begin{aligned}
& a_1= \max \{ n : f_{i_1}^n b \neq 0 \}, \\
&  a_2= \max \{ n : f_{i_2}^n f_{i_1}^{a_1} b \neq 0\},
\end{aligned}
\]
and so on. 
Record this as a word in the letters $I$ consisting of $a_1$ $i_1$'s, followed by $a_2$ $i_2$'s, and so on.  Sometimes we write this as
$$i_1^{a_1} i_2 ^{a_2} \cdots i_k ^{a_k}.$$
The string data uniquely determines an element $b \in B(-\infty)$. 

Indexing $B(-\infty)$ by $\sqcup \Irr \Lambda(\bfv)$, Theorem \ref{th:QVrealization} shows that the string data of $X \in\Irr \Lambda(\bfv)$ gives the dimensions of the graded socle filtration of a generic $T \in X$:
$$
\begin{aligned}
& a_1 = \dim \Hom({\Bbb C}_{i_1}, T), \\
& a_2= \dim \Hom({\Bbb C}_{i_2}, T/i_1 \text{socle}) ,\\
\end{aligned}
$$
and so on, where ${\Bbb C}_i$ is the one dimensional simple module in degree $i$. 

\begin{Remark} \label{rem:connected}
There are only non-trivial extensions between ${\Bbb C}_i$ and ${\Bbb C}_j$ if nodes $i$ and $j$ are connected in the Dynkin diagram. 
\end{Remark}

\subsection{Stablilty and root multiplicities}
Fix a stability condition $c$ so that, for any roots $\alpha$ and $\beta$, if $\arg c(\alpha)=\arg c(\beta)$ then $\beta$ and $\alpha$ are parallel. The following is the key to our method. It can be extracted from \cite{BKT:2014}, see \cite[Theorem 3.3]{T21} for the exact statement. 

\begin{Theorem} \label{th:BigK}
For any $\gamma \in Q_+$, the number of stable irreducible components of $\Lambda(\gamma)$ is the sum over all ways of writing $\gamma=v_1\beta_1+\cdots+v_n \beta_n$ as a sum of parallel roots $\beta_k$
of the product $m_{\beta_1} \cdots m_{\beta_n}$ of the corresponding root multiplicities. 
\end{Theorem}

\begin{Corollary}
If $\gamma$ is not parallel to any smaller root then the number of stable irreducible components of $\Lambda(\gamma)$ is exactly $m_\gamma$. Equivalently, $m_\gamma$ is the number of string data associated to stable irreducible components. This in particular holds if $\gamma = a_1 \alpha_1 + \cdots + a_k \alpha_k$ with $\text{gcd} (a_1, \ldots, a_k)=1$. 
\end{Corollary}

\section{Rank three and proof of Theorem \ref{th:main}}

\subsection{Setup}

 As in the introduction, consider the case of the Dynkin diagram
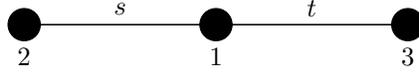
\begin{center}
\begin{tikzpicture}[scale=0.85]
\node [draw, circle, fill=black] at (0,0) {.};
\node [draw, circle, fill=black] at (3,0) {.};
\node [draw, circle, fill=black] at (6,0) {.};
\draw[line width = 0.02cm] (0,0)--(3,0);
\draw[line width = 0.02cm] (3,0)--(6,0);

\draw node at (0,-0.5){2};
\draw node at (3,-0.5){1};
\draw node at (6,-0.5){3};
\draw node at (1.5,0.25) {$s$};
\draw node at (4.5,0.25) {$t$};
\end{tikzpicture}
\end{center}
The corresponding Cartan matrix is
 $$
\left(
\begin{array}{rrr}
2&-s&-t\\
-s&2&0\\
-t&0&2
\end{array}
\right).
$$
Consider the charge defined by
$$c(\alpha_1)=-1+i, c(\alpha_2) =(s-\epsilon)(1+i), c(\alpha_3)= t(1 +i) $$
where $\epsilon$ is infinitesmal.  
Then $\arg (c(a_1 \alpha_1+b_1 \alpha_2 +c_1\alpha_3)) < \arg (c( a_2 \alpha_1+b_2 \alpha_2 +c_2\alpha_3))$ exactly if
$$\begin{cases}
\frac{sb_1+tc_1}{a_1} > \frac{sb_2+tc_2}{a_2}, \text{ or }  \\
\frac{sb_1+tc_1}{a_1} = \frac{sb_2+tc_2}{a_2} \text{ and }  \frac{b_1}{c_1} < \frac{b_2}{c_2}. 
\end{cases}
$$
If $\beta=a \alpha_1+b \alpha_2 +c\alpha_3$ and  $\text{gcd}(a,b,c)=1$ then Theorem \ref{th:BigK} shows that $m_\beta$ is the number of stable irreducible components of $\Lambda(\beta)$. We will prove Theorem \ref{th:main} by showing that the string data of each stable irreducible component satisfies all conditions in that theorem. Therefore counting words satisfying those conditions gives an over-estimate of the root multiplicity. 

\begin{Remark}
We could in principle use any charge to get an upper bound on root multiplicities. This one is convenient for technical reasons. 
\end{Remark}

\subsection{Two lemmas about root space}

\begin{Lemma} \label{lem:ref} 
Fix $\gamma =a\alpha_1 +b\alpha_2+c\alpha_3 \in Q$ with $a,b,c \geq 0$. 
\begin{itemize}
\item Assume $\frac{a}{sb+tc}<  \frac{1}{2}-\frac{\sqrt{(s^2+t^2)^2-4(s^2+t^2)}}{2(s^2+t^2)}$ and
$s_2s_3\beta= a\alpha_1 +b'\alpha_2+c'\alpha_3$ has $b',c' \geq 0$. Then
$\frac{a}{sb'+tc'}>  \frac{1}{2}+\frac{\sqrt{(s^2+t^2)^2-4(s^2+t^2)}}{2(s^2+t^2)}$.

\item Assume  $\frac{a}{sb+tc}>   \frac{1}{2}+\frac{\sqrt{(s^2+t^2)^2-4(s^2+t^2)}}{2(s^2+t^2)}$ and 
$s_1\beta= a'\alpha_1 +b\alpha_2+c\alpha_3$ has $a' \geq 0$. Then
$\frac{a'}{sb+tc}<  \frac{1}{2}-\frac{\sqrt{(s^2+t^2)^2-4(s^2+t^2)}}{2(s^2+t^2)}$.
\end{itemize}
\end{Lemma}

\begin{proof}
Consider the first case. By definition
$$s_2s_3\beta= a\alpha_1 +(sa-b)\alpha_2+(ta-c)\alpha_3.$$
Then
$$
\begin{aligned}
\frac{a}{s(sa-b)+t(ta-c)} &= \frac{a}{(s^2+t^2)a -(sb+tc)}  \\
&=\frac{1}{ s^2+t^2-\frac{sb+tc}{a}} \\
&\geq \frac{1}{ s^2+t^2-  \frac{1}{ \frac{1}{2}-\frac{\sqrt{(s^2+t^2)^2-4(s^2+t^2)}}{2(s^2+t^2)}}  }\\
&=  \frac{1}{2}+\frac{\sqrt{(s^2+t^2)^2-4(s^2+t^2)}}{2(s^2+t^2)},
\end{aligned}
$$
where the last step requires some simplification. The second case is similar. 
\end{proof}

\begin{Lemma} \label{lem:imcond} For any positive imaginary root $\beta=a \alpha_1+b\alpha_2+c\alpha_3$, 
 $$\frac{1}{2}-\frac{\sqrt{(s^2+t^2)^2-4(s^2+t^2)}}{2(s^2+t^2)}< \frac{a}{sb+tc}<  \frac{1}{2}+\frac{\sqrt{(s^2+t^2)^2-4(s^2+t^2)}}{2(s^2+t^2)}<1.$$
\end{Lemma}

\begin{proof}
Proceed by induction on $a+sb+tc$. 
The base case is trivial since if $a+sb+tc=1$ there are no imaginary roots. Fix a positive imaginary root $\beta$ with $a+sb+tc>1$. First assume the left inequality fails. Since reflection preserves the set of positive imaginary roots 
$s_2s_3\beta= a\alpha_1 +b'\alpha_2+c'\alpha_3$ 
has $b',c'\geq0$. Then, by Lemma \ref{lem:ref}
$s_2s_3\beta= a\alpha_1 +b'\alpha_2+c'\alpha_3$ satisfies $$\frac{a}{sb'+tc'}> \frac{1}{2}+\frac{\sqrt{(s^2+t^2)^2-4(s^2+t^2)}}{2(s^2+t^2)}.$$ This certainly also implies $a+sb'+tc'<a+sb+tc$ so by induction $s_2s_3 \beta$ is not an imaginary root, contradicting the assumption that $\beta$ is. The other case is similar. 
\end{proof}

\subsection{Specialized notation}

Fix $$T =A \oplus B \oplus C \in \Lambda(a\alpha_1+b\alpha_2+c\alpha_3).$$ 
Take string data as in \S\ref{ss:stringdata} using the sequence $1,2,3,1,2,3, \cdots$. We use the notation $a_1,b_1,c_1,a_2,b_2,c_2, \cdots$ for this string data, where,
$$
\begin{aligned} 
& a_1 = \dim \Hom ({\Bbb C}_1, T) \qquad  \qquad  \mbox{} \\
&b_1 = \dim \Hom ({\Bbb C}_2, T_1 ), \text{ where } T_1= T/1-\text{socle}\\
&c_1 = \dim \Hom ({\Bbb C}_3, T_2 ), \text{ where } T_2= T_1/2-\text{socle}\\
&a_2 = \dim \Hom ({\Bbb C}_1, T_3 ), \text{ where } T_3= T_2/3-\text{socle},
\end{aligned}
$$
and so on. 
Taking every third step in this socle filtration gives a filtration by $I$ graded submodules:
$$T= A_k \oplus B_k \oplus C_k \supset A_{k-1} \oplus B_{k-1} \oplus C_{k-1}  \supset \cdots \supset A_{1} \oplus B_{1} \oplus C_{1}.$$

\begin{Definition}
For each $j$, let $\bar A_j = A_j/A_{j-1}$. More generally, for any $1 \leq i\leq j \leq k$, let $\bar A_{ij} = A_j/A_{i-1}$. Then $\dim \bar A_j = a_j$ and $\dim \overline A_{ij}= a_i+ \cdots+ a_j$. Define subquotients for $B$ and $C$ similarly.
\end{Definition}

\begin{Remark} \label{sub-quotients}
For any $1 \leq i<j \leq k$, 
$$\bar A_{ij}+\bar B_{ij}+\bar C_{ij} \quad \text{and} \quad \bar A_{ij}+\bar B_{i-1,j-1}+\bar C_{i-1,j-1} $$ 
are both sub-quotients of $T$. 
\end{Remark}

\begin{Definition}
Fix $T \in \Lambda(a\alpha_1+b\alpha_2+c\alpha_3)$. Define the following maps:
$${}_1x_2 = \bigoplus_{\so(a)=2, \ta(a)=1} x_a: B \rightarrow A^s, \quad  {}_2x_1= \bigoplus_{\so(a)=1, \ta(a)=2} x_a: A \rightarrow B^s, $$
$${}_1x_3 = \bigoplus_{\so(a)=3, \ta(a)=1} x_a: C \rightarrow A^t, \quad  {}_3x_1= \bigoplus_{\so(a)=1, \ta(a)=3} x_a: A \rightarrow C^t, $$

$${}_1\bar x_2 = \sum_{\so(a)=2, \ta(a)=1} x_a: B^s \rightarrow A, \quad  {}_2\bar x_1= \sum_{\so(a)=1, \ta(a)=2} x_a: A^s \rightarrow B, $$
$${}_1\bar x_3 = \sum_{\so(a)=3, \ta(a)=1} x_a: C^t \rightarrow A, \quad  {}_3\bar x_1= \sum_{\so(a)=1, \ta(a)=3} x_a: A^t \rightarrow C.$$
For any $1 \leq i<j \leq k$ these maps make sense if $A,B,C$ are replaced in the domain with $\overline A_{ij}, \overline B_{ij}$ or $\overline C_{ij}$, and in the range by $\overline A_{ij}, \overline B_{i-1,j-1},$ and $\overline C_{i-1,j-1}$, simply by applying the definition to the sub-quotient representations from Remark \ref{sub-quotients}. We denote the resulting maps with a superscript of $i$ or $ij$. 
\end{Definition}
The preprojective relations are then
$$
{}_ 2\bar x_1 \circ {}_1x_2=0, \quad {}_1 \bar x_2 \circ {}_2x_1+{}_1 \bar x_3 \circ {}_3x_1=0, \quad {}_3 \bar x_1 \circ {}_1x_3=0.
$$

\subsection{Proof of the main theorem}
We will show that the conditions of Theorem \ref{th:main} in fact holds for all stable modules, from which the statement on stable components follows. So, fix a positive imaginary root $\beta =a \alpha_1+b\alpha_2+c\alpha_3$ and  a stable module $T \in \Lambda(a \alpha_1+b\alpha_2+c\alpha_3)$ with string data $a_1,b_1,c_1,a_2...$. Since there are no imaginary roots of the form $b \alpha_2+c\alpha_3$ we can assume $a \neq 0$. 

\begin{Lemma} \label{lem:main0}
For each $1\leq i \leq k$, $a_i \neq 0$ and $b_i$ or $c_i$ is also non-zero. 
\end{Lemma}

\begin{proof}
First assume some $a_i=0$. If $i>1$, Then two consecutive steps in the filtration are
$$A_{i-1} \oplus B_{i-1} \oplus C_{i-1}, \quad \text{and} \quad A_{i-1} \oplus B_{i} \oplus C_{i}.$$
Since there are no non-trivial extension of the simple modules ${\Bbb C}_2$ and ${\Bbb C}_3$ with themselves or each other, it follows that $B_i=B_{i-1}$ and $C_i=C_{i-1}$ as well. But then $$A_i \oplus B_i \oplus C_i = A_{i-1} \oplus B_{i-1} \oplus C_{i-1}$$
which is impossible.

If $i=1$ then this implies that either ${\Bbb C}_2$ or ${\Bbb C}_3$ is a submodule of $T$. Since $a \neq 0$ this contradicts stability. 

The proof that $b_i$ or $c_i$ is also non-zero for each $i$ is similar. 
\end{proof}

\begin{Lemma} \label{lem:Dyck}
Draw a path in the plane by drawing each 1 as a vertical line of length 1, each 2 as a horizontal line of length $s$, and each 3 as a horizontal line of length $t$. The result is a rational Dyck path. That is, a straight line connecting the beginning and end of the path stays weakly below the path. Furthermore, for any $k$ where the path touches the diagonal, $\frac{b_1+\cdots +b_k}{c_1+\cdots+c_k} \geq  \frac{b}{c}$.
\end{Lemma}

\begin{proof}
Each left-prefix of the sequence corresponds to a sub-model of $T$ and if the conditions do not hold then this gives a submodule that violates stability. 
\end{proof}

\begin{Example} Consider the case $s=2,t=1$. 
$$11223122312$$
becomes the path
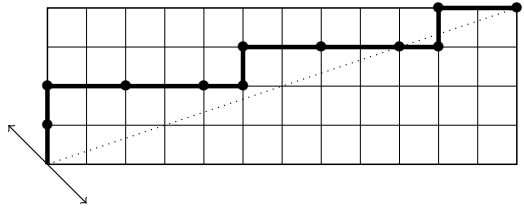
\begin{center}
\begin{tikzpicture}[scale=0.52]

\draw[line width= 0.02cm] (0,0)--(12,0)--(12,4)--(0,4)--(0,0);
\draw[line width= 0.01cm] (0,1)--(12,1);
\draw[line width= 0.01cm] (0,2)--(12,2);
\draw[line width= 0.01cm] (0,3)--(12,3);
\draw[line width= 0.01cm] (1,0)--(1,4);
\draw[line width= 0.01cm] (2,0)--(2,4);
\draw[line width= 0.01cm] (3,0)--(3,4);
\draw[line width= 0.01cm] (4,0)--(4,4);
\draw[line width= 0.01cm] (5,0)--(5,4);
\draw[line width= 0.01cm] (6,0)--(6,4);
\draw[line width= 0.01cm] (7,0)--(7,4);
\draw[line width= 0.01cm] (8,0)--(8,4);
\draw[line width= 0.01cm] (9,0)--(9,4);
\draw[line width= 0.01cm] (10,0)--(10,4);
\draw[line width= 0.01cm] (11,0)--(11,4);
\draw[line width= 0.06cm] (0,0)--(0,2)--(5,2)--(5,3)--(10,3)--(10,4)--(12,4);
\draw[line width= 0.015cm, dotted] (0,0)--(12,4);

\draw[line width= 0.015cm, <->] (-1,1)--(1,-1);

\draw node at (0,1){$\bullet$};
\draw node at (0,2){$\bullet$};
\draw node at (2,2){$\bullet$};
\draw node at (4,2){$\bullet$};
\draw node at (5,2){$\bullet$};
\draw node at (5,3){$\bullet$};
\draw node at (7,3){$\bullet$};
\draw node at (9,3){$\bullet$};
\draw node at (10,3){$\bullet$};
\draw node at (10,4){$\bullet$};
\draw node at (12,4){$\bullet$};
\end{tikzpicture}
\end{center}
which is not a rational Dyck path because it dips below the diagonal at $(10,3)$.
To see the relationship with stability, overly this picture on the complex plane oriented so that the real axis is at 45 degrees to the picture as shown. Then each node shown along the path is $c(S)$ for the submodule $S$ corresponding to a left subword $s$. The stability condition is violated because one of the nodes is below the diagonal so the charge of that submodule has a smaller argument than the charge of the whole module. 
\end{Example}

\begin{Example} Again consider the case $s=2,t=1$. Both
$$\quad 112331122  \quad \text{and} \quad 112211233 $$
correspond to the Dyck path
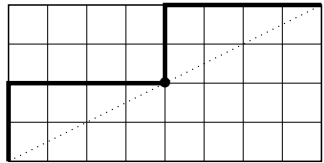
\begin{center}
\begin{tikzpicture}[scale=0.52]

\draw[line width= 0.02cm] (0,0)--(8,0)--(8,4)--(0,4)--(0,0);
\draw[line width= 0.01cm] (0,1)--(8,1);
\draw[line width= 0.01cm] (0,2)--(8,2);
\draw[line width= 0.01cm] (0,3)--(8,3);
\draw[line width= 0.01cm] (1,0)--(1,4);
\draw[line width= 0.01cm] (2,0)--(2,4);
\draw[line width= 0.01cm] (3,0)--(3,4);
\draw[line width= 0.01cm] (4,0)--(4,4);
\draw[line width= 0.01cm] (5,0)--(5,4);
\draw[line width= 0.01cm] (6,0)--(6,4);
\draw[line width= 0.01cm] (7,0)--(7,4);
\draw[line width= 0.01cm] (8,0)--(8,4);
\draw[line width= 0.06cm] (0,0)--(0,2)--(4,2)--(4,4)--(8,4);
\draw[line width= 0.015cm, dotted] (0,0)--(8,4);

\draw node at (4,2){$\bullet$};

\end{tikzpicture}
\end{center}
The submodule of $112331122$ corresponding to $11233$ fails stability because it touches the diagonal and has a lower ratio of 2s to 3s than the whole word. 
 $112211233$ does correspond to a stable component since $1122$ has a higher ratio of 2s to 3s. Equivalently, since $2|c(\alpha_3)|$ is slightly larger than $|c(\alpha_2)|$, in the first case the node is actually slightly to the right of the diagonal whereas in the second case it is slightly to the left. 
\end{Example}

\begin{Lemma} \label{lem:main3easy} For every $1 \leq i \leq k$, 
$b_i \leq sa_i$ and  $c_i \leq ta_i$.
\end{Lemma}

\begin{proof}
Assume $T$ has this string data and $b_i>sa_i$. Then by dimension count ${}_1x_2: \overline B_i \rightarrow \overline A_i^{\oplus s}$ must have a kernel. If $i>1$ this contradicts the definition of the socle filtration, and if $i=1$ it contradicts stability. Thus $b_i \leq sa_i$  for all $i$. Similarly  $c_i \leq ta_i$.  
\end{proof}

\begin{Lemma}   \label{lem:mainn}
Fix $i<k$ and let
$$n_{Bi}=  \min \{ b_i, s a_{i+1}-b_{i+1}\}, \quad n_{Ci}=  \min \{ c_i, t a_{i+1}-c_{i+1}\}.$$
Then 
 $$a_{i+1} \leq s n_{Bi}+tn_{Ci}- \max \{s^{-1} n_{Bi}, t^{-1} n_{Ci} \}.$$ 
\end{Lemma}

\begin{proof}
By the definition of the socle filtration, ${}_1 x^{i+1}_2: \overline B_{i+1} \rightarrow \overline A_{i+1}^{\oplus s}$ is injective. Also, 
by the preprojective relations, 
$${}_2 \bar x^i_1  \circ {}_1 x^{i+1}_2: \overline B_{i+1} \rightarrow \overline  B_i$$ 
is the zero map. Hence $$\dim \im {}_2 \bar x^i_1 \leq \dim \overline A_{i+1}^{\oplus s} - \dim \overline B_{i+1} = s a_{i+1}-b_{i+1}.$$
Certainly  $$\dim \im {}_2 \bar x^i_1 \leq \dim \overline B_i=b_i.$$ Putting this together,  $$\dim \im  {}_2 \bar x^i_1 \leq n_{Bi}.$$
Similarly, $\dim  \im  {}_3 \bar x^i_1 \leq n_{Ci}$.

Let $I_b = \im {}_2 \bar x^i_1 \subset \overline  B_i$ and $I_c = \im {}_3 \bar x^i_1 \subset \overline  C_i.$ These have dimensions at most $n_{Bi}$ and $n_{Ci}$ respectively, so define $e_{Bi},e_{Ci} \geq 0$ so that
$$
\dim I_b= n_{Bi}-e_{Bi} \quad \text{and} \quad \dim I_c= n_{Ci}-e_{Ci} 
$$
Again using the definition of a socle filtration, $${}_2x^i_1 \oplus {}_3x^i_1 : \overline A_{i+1} \rightarrow I_b^{\oplus s} \oplus I_c^{\oplus t}$$ is injective, 
so the dimension of its image is $a_{i+1}$.

Now, ${}_1 x^i_2|_{I_b}$ is injective, so its image has dimension $n_{Bi} - e_{Bi}$. Recalling that 
$${}_1 x^i_2|_{I_b}= \bigoplus_{a : \so(a)=2, \ta(a)=1} x^i_a,$$
at least one of these $x^i_a$ must have $\dim \im x^i_a|_{I_b} > s^{-1}(n_{Bi}-e_{Bi}).$ This in turn implies that  $\dim \im {}_1 \bar x^i_2|_{I^s_b} \geq s^{-1} (n_{Bi}-e_{Bi}).$ Similarly, 
 $\dim \im {}_1 \bar x^i_3|_{I_c^t} \geq t^{-1} (n_{Ci}-e_{Ci}).$
 So, 
 $$\dim \im  {}_1 \bar x^i_2 + {}_1 \bar x^i_3|_{I_b^s+I_c^t} \geq \max \{ s^{-1} (n_{Bi}-e_{Bi}) , t^{-1} (n_{Ci}-e_{Ci})\}. $$ 
 Using the preprojective relation again, 
 $$( {}_1 \bar x^i_2 +  {}_1 \bar x^i_3)  \circ  ({}_2x^i_1 \oplus   {}_3x^i_1) =0.$$
 Hence 
$$ 
\begin{aligned}
a_{i+1} & = \dim \im ( {}_2x^i_1 \oplus   {}_3x^i_1) \\ 
& \leq  \dim \ker  (  {}_1 \bar x^i_2 + {}_1 \bar x^i_3 )|_{I_b^s+I_c^t} \\
& = \dim (I_b^s+I_c^t) -  \dim \im  (  {}_1 \bar x^i_2 + {}_1 \bar x^i_3 )|_{I_b^s+I_c^t} \\
& \leq  s  (n_{Bi}-e_{Bi})+ t (n_{Ci}-e_{Ci}) -  \max \{ s^{-1}  (n_{Bi}-e_{Bi}) , t^{-1} (n_{Ci}-e_{Ci})\} \\
&\leq  s  n_{Bi} + t n_{Ci}-s e_{Bi}-te_{Ci} -  \max \{ s^{-1}  n_{Bi} , t^{-1} n_{Ci}\}  +  \max \{ s^{-1}  e_{Bi} , t^{-1} e_{Ci}\}\\
&\leq  s  n_{Bi} + t n_{Ci} -  \max \{ s^{-1}  n_{Bi} , t^{-1} n_{Ci}\},  \\
\end{aligned}
$$
where the last step follows because $s,t \geq 1$ and $e_{Bi}, e_{Ci} \geq 0$. 
\end{proof}

\begin{Lemma} \label{lem:sub-quotient} 
Fix $1< i \leq k$. 
\begin{itemize}
\item Assume that, for some $i$, $\overline A_i \oplus \overline B_i \oplus \overline C_i$ has a submodule of dimension $a',b',c'$ with $\frac{a'}{sb'+tc'}<  \frac{1}{2}-\frac{\sqrt{(s^2+t^2)^2-4(s^2+t^2)}}{2(s^2+t^2)}$. Then $\overline A_i \oplus \overline B_{i-1} \oplus \overline C_{i-1}$ has a submodule of dimension $a', b'',c''$ with $\frac{a'}{sb''+tc''}>  \frac{1}{2}+\frac{\sqrt{(s^2+t^2)^2-4(s^2+t^2)}}{2(s^2+t^2)}$.

\item  Assume that, for some $i$, $\overline A_i \oplus \overline B_{i-1} \oplus \overline C_{i-1}$ has a submodule of dimension $a',b',c'$ with $\frac{a'}{sb'+tc'}> \frac{1}{2}+\frac{\sqrt{(s^2+t^2)^2-4(s^2+t^2)}}{2(s^2+t^2)}$. Then $\overline A_{i-1} \oplus \overline B_{i-1} \oplus \overline C_{i-1}$ has a submodule of dimension $a'', b',c'$ with $\frac{a''}{sb'+tc'}<  \frac{1}{2}-\frac{\sqrt{(s^2+t^2)^2-4(s^2+t^2)}}{2(s^2+t^2)}$.
\end{itemize}
\end{Lemma}

\begin{proof}
Consider the first case.  By the preprojective relation
$$ {}_2 \bar x^i_1 \circ {}_1 x^i_2 =0.$$
The map $ {}_1 x^i_2 $ is injective, so $0 \leq \dim \im  {}_2 \bar x^i_1 \leq sa'-b'$.
Similarly, $0 \leq \dim \im  {}_3 \bar x^i_1 \leq ta'-c'$. This implies the existence of a submodule of  $\overline A_i \oplus \overline B_{i-1} \oplus \overline C_{i-1}$ of dimension $a',b'',c''$ with 
$$\frac{a'}{sb''+tc''}>\frac{a'}{s(sa'-b')+t(ta'-c')}.$$
This last fraction is the ratio for $s_2s_3 (a'\alpha_1+b'\alpha_2+c'\alpha_3),$ so is larger than 
$\frac{1}{2}+\frac{\sqrt{(s^2+t^2)^2-4(s^2+t^2)}}{2(s^2+t^2)}$ by Lemma \ref{lem:ref}.

The other case is similar. 
\end{proof}

\begin{Lemma} \label{lem:mainr}
For $1 \leq i <k$ and any submodule $\overline A'_{i+1} \oplus \overline B'_i\oplus \overline C'_i$ of  $ \overline A_{i+1} \oplus \overline B_i\oplus  \overline C_i $ with $\dim \overline A'_{i+1}=a_{i+1}', \dim \overline B'_i =b'_i, \dim  \overline C'_i =c'_i$, 
\begin{equation} \label{eq:sm} \frac{a'_{i+1}}{sb'+tc'} \leq \frac{1}{2}+\frac{\sqrt{(s^2+t^2)^2-4(s^2+t^2)}}{2(s^2+t^2)}.
\end{equation}
In particular, 
\begin{equation} \label{eq:sth}
\frac{a_{i+1}}{sn_{Bi}+tn_{Ci}} \leq \frac{1}{2}+\frac{\sqrt{(s^2+t^2)^2-4(s^2+t^2)}}{2(s^2+t^2)}.
\end{equation}
\end{Lemma}

\begin{proof}
Assume there is a sub-module where \eqref{eq:sm} fails. 
If $i=1$ then, by Lemma \ref{lem:sub-quotient}, 
$ A_1 \oplus B_1 \oplus C_1$ has a submodule $A_1'\oplus B_1'\oplus C_1'$ of dimension $(a_1',b_1',c_1')$ with 
 $$\frac{a_1'}{sb_1'+tc_1'}  < \frac{1}{2}-\frac{\sqrt{(s^2+t^2)^2-4(s^2+t^2)}}{2(s^2+t^2)}.$$
Since $a \alpha_1+b\alpha_2+c\alpha_3$ is imaginary, Lemma \ref{lem:imcond} shows that
$$ \frac{a}{sb+tc}> \frac{1}{2}-\frac{\sqrt{(s^2+t^2)^2-4(s^2+t^2)}}{2(s^2+t^2)}.$$
Taking reciprocals we see that  $A_1'\oplus B_1'\oplus C_1'$ violates stability. 

If $i>1$ then applying Lemma \ref{lem:sub-quotient} twice gives a sub-quotient for a lower $i$ which violates the same condition, and one proceeds by induction. 

To see \eqref{eq:sth}, notice that, as in the proof of Lemma \ref{lem:mainn}, the preprojective relation implies that the image of ${}_2 \overline x_1 \oplus {}_3 \overline x_1; \overline A_{i+1}^{s+t} \rightarrow  \overline B_i\oplus \overline C_i$ has dimension at most $(n_{Bi}, n_{Ci})$. Hence if this equation is violated it immediately gives a sub-quotient violating \eqref{eq:sm}.
\end{proof}

\begin{proof}[Proof of  Theorem \ref{th:main}] 
A generic module in a stable irreducible component is stable, and by definition the string data of a generic module is the string data of the component. Theorem \ref{th:main} gives conditions that must be satisfied by a stable component. To prove that these conditions hold, it suffices to prove that they must hold for the string data of any stable module. In this way:

Lemma \ref{lem:main0} shows \eqref{main0}. 

Lemma \ref{lem:Dyck} shows \eqref{main1} and \eqref{main2}.

Lemma \ref{lem:main3easy} shows \eqref{part:main3easy-s}.

Lemma \ref{lem:mainn} shows \eqref{part:combo}.

Lemma \ref{lem:mainr} shows \eqref{part:ratio}.
\end{proof}

\section{Examples}

\subsection{The case $s=2,t=1$ }
In this case many root multiplicities can be found in \cite[Chapter 11]{Kac:1990}. To estimate the multiplicity of
$\beta= a \alpha_1+ b\alpha_2+c \alpha_3$ with $\text{gcd}(a,b,c)=1$, Theorem \ref{th:main} says we should count words  $1^{a_1} 2^{b_1}3^{c_1} \cdots 1^{a_k} 2^{b_k}3^{c_k}$ in 1,2,3 such that 
\begin{itemize}

\item  For each $i$, $a_i \neq 0$ and $b_i$ or $c_i$ is also non-zero.

\item The resulting path is a rational Dyck path.

\item  If a prefix $1^{a_1} 2^{b_1}3^{c_1} \cdots 1^{a_k} 2^{b_k}3^{c_k}$ corresponds to a point where the Dyck path touches the diagonal, then $\displaystyle  \frac{b_1+\cdots +b_k}{c_1+\cdots+c_k} >  \frac{b}{c}.$

\item $\displaystyle \frac{b_i}{a_i} \leq 2$, $\displaystyle \frac{c_i}{a_i} \leq 1$

\noindent $\mbox{}$ \hspace{-0.3in} Let $n_{Bi}=  \min \{ b_i, 2a_{i+1}-b_{i+1}\}, \ n_{Ci}=  \min \{ c_i,  a_{i+1}-c_{i+1}\}.$ Then 

\item  $\displaystyle a_{i+1} \leq 2 n_{Bi}+n_{Ci}- \max \{\frac{n_{Bi}}{2} , n_{Ci} \}$ \ \text{ and} 

\item $\displaystyle \frac{a_{i+1}}{2n_{Bi}+n_{Ci}} \leq \frac{1}{2}+\frac{\sqrt{5}}{10} \simeq 0.7236$.

\end{itemize}

\noindent The resulting estimates can be found in the Figure \ref{fig:21} for many roots.

\begin{Example}
For $\beta=4\alpha_1+3 \alpha_2+2\alpha_3$, six paths satisfy these conditions:

112312123

112123123

111223123

112211233

112311223

111122233

\noindent However, the root multiplicity is $5$. The word which does not correspond to a valid stable component is 
11{\color{red}2311}223. 
The reason is that the sub-quotient $Q$ corresponding to the sub-string 2311 (shown in red) has the property that the ${\Bbb C}_2$ and ${\Bbb C}_3$ together have only the freedom to map to a single copy of ${\Bbb C}_1$. This implies the existence of a submodule with socle filtration 123 which violates stability. This path will be eliminated by the refined conditions in the next section. 
\end{Example}

\subsection{The case $s=t=2$} .
To estimate the multiplicity of
$\beta= a \alpha_1+ b\alpha_2+c \alpha_3$, Theorem \ref{th:main} now says we should count words in 1,2,3 such that 
\begin{itemize}

\item  For each $i$, $a_i \neq 0$ and $b_i$ or $c_i$ is also non-zero.

\item The resulting path is a Dyck path.

\item  If a prefix $1^{a_1} 2^{b_1}3^{c_1} \cdots 1^{a_k} 2^{b_k}3^{c_k}$ corresponds to a point where the Dyck path touches the diagonal, then $\displaystyle  \frac{b_1+\cdots +b_k}{c_1+\cdots+c_k} >  \frac{b}{c}.$

\item $\displaystyle \frac{b_i}{a_i}, \frac{c_i}{a_i} \leq 2$

\noindent $\mbox{}$ \hspace{-0.3in} Let $n_{Bi}=  \min \{ b_i, 2a_{i+1}-b_{i+1}\}, \ n_{Ci}=  \min \{ c_i,  2a_{i+1}-c_{i+1}\}.$ Then 

\item  $\displaystyle a_{i+1} \leq 2 n_{Bi}+2 n_{Ci}- \max \{\frac{n_{Bi}}{2} , \frac{n_{Ci}}{2} \}.$

\item $\displaystyle \frac{a_{i+1}}{2n_{Bi}+2n_{Ci}} \leq \frac{1}{2}+\frac{\sqrt{2}}{4}$.

\end{itemize}

\noindent The resulting estimates can be found in the Figure \ref{fig:22} for many roots. We now consider a few cases where our estimate is above the actual multiplicity, and which motivate the conditions in the next section.

\begin{Example} 
The path 
	\begin{equation}
		\label{eq:eea}
		1^3 2^3  {\color{red}1^32^2 3 1^4} 2^5
	\end{equation}
	satisfies all the conditions in Theorem \ref{th:main}. However, by the pre-projective relation and the fact that this is a socle filtration, the sub-quotient corresponding to the sub-path $1^3 2^231^4$ (shown in red) has $\dim \im ({}_1x_2 +{}_1x_3) \leq 2$. This then implies the existence of a submodule 
	$$1^32^31^2 2^2 3,$$ 
	and that violates stability. This is caught by Theorem \ref{th:cond2} below with $i=j=2$.
\end{Example}

\begin{Example}
	Now consider 
	$$\underline{1^4 2^3} {\color{red}  1^3\underline{2^2 3 } 1^4\underline{ 2^2 3 } 1^4} 2^6 1^6 2^8.$$
	By the preprojective relation, the sub-quotient corresponding to $2^2 3 1^42^2 3 1^4$ has $\dim \im (x_2 +x_3) \leq 4$. 
	Thus the vectors corresponding to the underlined part of the path must generate a sub-module violating stability. 
	This is caught by Theorem \ref{th:cond2} with $i=2, j=3$ and motivates why we need to consider cases where $i\neq j$ where we are looking at a longer segment of the word.  
\end{Example}

\begin{Example}
	Now take 
	$$1^42^41^42^331^42^61^52^6.$$
	Then the submodule corresponding to $2^331^42^6$ has $\dim \im ({}_2x_1) \leq 2,$ which forces a submodule of the form 
	$$1^42^4{\color{red} 1^42^231^4}.$$
	But now the sub-quotient $ 1^42^231^4$ has $\dim \im ({}_1 x_2+ {}_1 x_3) \leq 2$ 
	forcing a submodule with data
	$$1^42^41^22^23,$$
	and now stability has been violated. 
	The cases when $n_{Bij}= sn_{Aij}-(b_{i+1}+b_{i+2}+\cdots+b_{j+1})$ (or similarly with $n_{Cij}$) in Theorem \ref{th:cond2} catch this sort of problem. 
\end{Example}

\section{Refined conditions}

In \cite{T21} two types of conditions are given: \cite[Theorem 4.3]{T21} gives local conditions that must be satisfied anywhere along the path of a stable component, and \cite[Theorem 4.9]{T21} gives extra conditions that must be satisfied near the beginning of the path and at places where it is near the diagonal. Our Theorem \ref{th:main} gives conditions analogous to \cite[Theorem 4.3]{T21}. We now give some conditions analogous to \cite[Theorem 4.9]{T21}.

\begin{Theorem} \label{th:cond2}
Fix  $\beta= a\alpha_1+b\alpha_2+c \alpha_3$ and a path $1^{a_1}2^{b_1} 3^{c_1} \cdots 1^{a_k}2^{b_k}3^{c_k}$ with
$a_1+\cdots + a_k=a, b_1+\cdots + b_k=b, c_1+\cdots + c_k=c.$ For any  
 $ 1\leq i \leq j < k$, let
\begin{itemize}
    \item $n_{Aij}=a_{i+1}+a_{i+2}+\cdots+a_{j+1}$
    \item $n_{Bij}=\min(b_i+b_{i+1}+\cdots+b_{j},sn_{Aij}-(b_{i+1}+b_{i+2}+\cdots+b_{j+1}))$
    \item $n_{Cij}=\min(c_i+c_{i+1}+\cdots+c_{j}, tn_{Aij}-(c_{i+1}+c_{i+2}+\cdots+c_{j+1}))$
    \vspace{.25cm} 
    \item $\tilde a_{ij}=a_1+a_2+\cdots+a_{i-1}+sn_{Bij}+tn_{Cij}-n_{Aij}$
    \item $\tilde b_{ij}=b_1+b_2+\cdots+b_{i-1} + n_{Bij} $
    \item $\tilde c_{ij}=c_1+c_2+\cdots+c_{i-1}+ n_{Cij}$
\end{itemize}
If this path is the string data of a stable component, then 
\begin{enumerate}
    \item $\frac{\tilde a_{ij}}{s\tilde b_{ij}+t\tilde c_{ij}} \geq \frac{a}{sb+tc}$,
    \item If $\frac{\tilde a_{ij}}{s\tilde b_{ij}+t\tilde c_{ij}} = \frac{a}{sb+tc}$ and $\tilde c_{ij}>0$, then $\frac{\tilde b_{ij}}{\tilde c_{ij}} \geq \frac{b}{c}$.
\end{enumerate}
\end{Theorem}

\begin{proof}
Fix a stable $\Lambda$-module $T=A\oplus B \oplus C$ with this string data, which must exist if this corresponds to a stable component. Fix $i$ and $j$. 

By the definition of the socle filtration, the map
$${}_1x^{i+1,j+1}_2 :  \overline B_{i+1,j+1}  \rightarrow \overline A_{i+1,j+1}^s$$
is injective. By the preprojective relation,
$$ {}_2 \bar x^{i+1,j+1}_1 \circ {}_1x^{i+1,j+1}_2 :  \overline B_{i+1,j+1} \rightarrow  \overline B_{i,j}$$
is the zero map, 
so 
$$\dim \im  {}_2 \bar x^{i+1,j+1}_1 = \dim \overline A_{i+1,j+1}^s - \dim \ker {}_2 \bar x^{i+1,j+1}_1  \leq s(a_{i+1}+\cdots+a_{j+1})-(b_{i+1}+\cdots+b_{j+1}).$$
Looking at the dimension of the target space, $\dim \im  {}_2 \bar x^{i+1,j+1}_1 \leq b_i+b_{i+1}+\cdots+b_{j}$. Together this means $\dim \im  {}_2 \bar x^{i+1,j+1}_1 \leq n_{Bij}$. Similarly, 
$\dim \im  {}_3 \bar x^{i+1,j+1}_1 \leq n_{Cij}.$ Define $e_{Bij}$ and $e_{Cij}$ by
$$
\dim  \im  {}_2 \bar x^{i+1,j+1}_1 =  n_{Bij} - e_{Bij}, \quad \text{and} \quad \dim \im  {}_3 \bar x^{i+1,j+1}_1  = n_{Cij} - e_{Cij}.
$$

Let $\tilde B_{ij} \subset B_j$ be the preimage of $\im  {}_2 \bar x^{i+1,j+1}_1$ in the quotient $\overline B_{ij}= B_j/B_{i-1}$ and $\tilde C_{ij} \subset C_j$ be the preimage of $\im  {}_3 \bar x^{i+1,j+1}_1$ in $\overline C_{ij}= C_j/C_{i-1}.$ Then
$$
\dim \tilde B_{ij} = \tilde b_{ij} - e_{Bij}, \quad \text{and} \quad \dim \tilde C_{ij} = \tilde c_{ij} - e_{Cij}.
$$

By the preprojective relation, the kernel of
$${}_1 \bar x^{ij}_2 \oplus {}_1 \bar x^{ij}_3 : (\im  {}_2 \bar x^{i+1,j+1}_1)^s \oplus (\im  {}_3 \bar x^{i+1,j+1}_1)^t \rightarrow \overline A_{ij}$$
contains $\im ( {}_2  x^{i+1,j+1}_1 \oplus {}_3 x^{i+1,j+1}_1)$, and by the definition of the socle filtration this map is injective so its image has dimension $n_{Aij}.$ Thus
$$\dim \im ( {}_1 \bar x^{ij}_2 \oplus {}_1 \bar x^{ij}_3 )|_{ (\im  {}_2 \bar x^{i+1,j+1}_1)^s \oplus (\im  {}_3 \bar x^{i+1,j+1}_1)^t }\leq sn_{Bij} - se_{Bij}+t n_{Cij} - te_{Cij} - n_{Aij}.$$

Let $\tilde A_{ij}$ by the preimage of $\im ( {}_1 \bar x^{ij}_2 \oplus {}_1 \bar x^{ij}_3 )|_{ (\im  {}_2 \bar x^{i+1,j+1}_1)^s \oplus (\im  {}_3 \bar x^{i+1,j+1}_1)^t } \subset  \overline A_{ij}$ in $\overline A_{ij}= A_j/A_{i-1}$. Then 
$$\dim \tilde A_{ij} \leq  sn_{Bij} - se_{Bij}+t n_{Cij} - te_{Cij} - n_{Aij} + a_1+ \cdots + a_{j-1}= \tilde a_{ij} - se_{Bij} - te_{Cij}.$$ 

By construction, 
$$\tilde M= \tilde A_{ij} \oplus \tilde B_{ij} \oplus \tilde C_{ij}$$
is a submodule of $T$. The calculations above show
\begin{equation} \label{eq:nd}
\frac{\dim \tilde A_{ij}}{s \dim \tilde B_{ij}+t \dim \tilde C_{ij}} \leq  \frac{\tilde a_{ij} - se_{Bij} - te_{Cij}}{s\tilde b_{ij} +t\tilde c_{ij}  -s e_{Bij} - te_{Cij}}. 
\end{equation}

By stability
\begin{equation} 
\frac{\dim \tilde A_{ij}}{s \dim \tilde B_{ij}+t \dim \tilde C_{ij}} \geq \frac{a}{sb+tc},
\end{equation}
so
\begin{equation} 
\frac{\tilde a_{ij} - se_{Bij} - te_{Cij}}{s\tilde b_{ij} +t\tilde c_{ij}  -s e_{Bij} - te_{Cij}} \geq \frac{a}{sb+tc}.
\end{equation}
We know $\frac{a}{sb+tc}<1$, so it follows that 
$$
\frac{\tilde a_{ij} }{s\tilde b_{ij} +t\tilde c_{ij}  } \geq \frac{a}{sb+tc}.
$$

To get equality $e_{Bij}$ and $e_{Cij}$ must both be zero, and then by stability $\frac{\tilde b_{ij}}{\tilde c_{ij}} \geq \frac{b}{c}$.
\end{proof}

\section{More examples}

\subsection{The case $s=t=2$ revisited}

\begin{Example}
The smallest example in Figure \ref{fig:22} where the number of paths satisfying both Theorems \ref{th:main} and \ref{th:cond2} is not the root multiplicity is $6 \alpha_1+5 \alpha_2+\alpha_3$, where the multiplicity is $21$ but there are 22 paths. The path which appears in our list but which does not correspond to a stable module is 
$$1^32^21^22312^2.$$
It is not too difficult to verify that this satisfies all of our conditions (or you can use our code). To see why it does not correspond to a stable module, notice that
\begin{itemize}
	\item Because of the preprojective relation at node 2, the rightmost $12^2$ implies that this $1$ has no freedom to map to anything in degree 2. So there is a submodule with data $1^3{\color{red} 2^21^231}.$
	
	\item Now the sub-quotient with data $ 2^21^231$ (colored red) has the property that, by dimension count and the preprojective relation, the $2s$ and $3s$ only map to 3 dimensions in degree 1. This implies they generate a module of dimension $3\alpha_1+2\alpha_2+\alpha_3$, which violates stability (because of the smaller ratio of 2's to 3's). 

\end{itemize}
\end{Example}

\subsection{The case $s=2$, $t=1$ revisited}

\begin{Example} \label{ex:33}
In this case
the smallest root on our list where the number of paths satisfying both Theorems \ref{th:main} and \ref{th:cond2} is not the root multiplicity is $6 \alpha_1+5\alpha_2+3\alpha_3$. Somewhat surprisingly, the correct multiplicity in 30 but there are 33 paths in this case. So, the first error is off by 3! The three paths that pass all the conditions but do not correspond to stable irreducible components are
\begin{align}
\label{eq:P33a} 1122112{\color{red}3123123} \\
\label{eq:P33c} 11212123123123 \\
\label{eq:P33d} 11122211233123
\end{align}
To see why these do not correspond to stable components, consider \eqref{eq:P33a}. Look at the quotient module corresponding to the right subword $3123123$ (in red). The map ${}_1 x_3$ has a 2 dimensional image by the definition of the socle filtration, so the  preprojective relation ${}_3 x_1 \circ {}_1 x_3=0$ implies that the map ${}_3 x_1 $ is the zero map on this quotient. So, there is a quotient module $Q$ isomorphic to ${\Bbb C}_3^3$ (which is in fact all of $C$), and  a submodule $P$ with data
$$112211{\bf2121}2.$$
Now, the sub-quotient marked in bold implies the existence of a submodule $P'$ with data 
$$11221212$$
Denote the span of these four $1$s by $W$. Then
the map ${}_1 x_3: Q \rightarrow A/W$ has a one dimensional kernel $K$ by dimension count. Then 
$P'\oplus K$ violates stability. 

The arguments for the other two cases are very similar: in both cases, there is a quotient $Q$ isomorphic to ${\Bbb C}_3^3$ by essentially the same argument, and then the argument that this implies a stability violation proceeds in the same way. 

\end{Example}

One may ask, why is the first example off by three? It seems there is no good reason, as there are smaller examples where the error is 1 if we look at roots which are not minimal, in that they can be reflected to smaller roots. Our method works just fine for such roots, although considering them directly is usually redundant. but let's consider such a case. 

\begin{Example}\label{ex:clean-nonminimal} Take $s=2, t=1$ and
the root $4\alpha_1+3 \alpha_2 +3 \alpha_3$. This can be reflected to  $4\alpha_1+3 \alpha_2+  \alpha_3$ and then to  $3\alpha_1+3 \alpha_2 + \alpha_3$, a root which we know has multiplicity 3. However, let's do the calculation with our methods directly on $4\alpha_1+3 \alpha_2 +3 \alpha_3$. Then a total of 5 paths pass the conditions in Theorem \ref{th:main}:

1 1 2 3 3 1 2 1 2 3

1 1 1 2 2 3 3 1 2 3

1 1 1 1 2 2 2 3 3 3

{\color{red} 1 1 2 3 1 2 3 1 2 3}

{\color{blue} 1 1 2 3 3 1 1 2 2 3}

\noindent The path 1123311223 is ruled out by Theorem \ref{th:cond2} with $i=j=1$. Here 
\begin{itemize}
	\item $n_{Aij}=2$ which is just $a_2$,
	\item $n_{Bij}=1$ since $\min \{ b_1, a_2-2b_2 \} = b_1=1$, 
	\item  $n_{Cij}= 1$ because $\min \{ c_1, a_2-c_2 \} = a_2-c_2=1,$ 
	\item $\tilde a_{ij}= 2 n_{Bij}+ n_{Cij}- n_{Aij}= 2+1-2=1,$
	\item $\tilde b_{ij}= n_{Bij}=1$,
	\item $\tilde c_{ij}= n_{Cij}=1$,
\end{itemize}
where the last three are simpler than in general because $i=1$ so we are at the beginning of the path.
The condition says 
$$\frac{\tilde a_{ij}}{2\tilde b_{ij}+\tilde c_{ij}} \geq \frac{4}{2\times 3+3},$$
which gives $\frac{1}{3} \geq \frac{4}{9}$, which is false. 

The other path which does not correspond to a stable component is 1123123123.
One can see that this is not stable by the same argument as in Example \ref{ex:33}, and in fact the situation is a little simpler here. Specifically:
\begin{itemize}
	\item In the quotient module corresponding to $123123$ the map ${}_1x_{3}$ must be injective by the definition of the socle filtration. By the pre-projective relation this implies that the map ${}_3x_1$ is actually trivial on the whole module, so $C$ is in fact a quotient module. 
	\item There is then a submodule with data $11{\bf 2121}2$. By a reflection the sub-quotient in bold implies the existence of a sub-module $S$ with data $1212$. 
	\item Let $Q= A' \oplus B' \oplus C$
	be the quotient by this submodule. Since $C$ is three dimensional and $A'$ is only 2 dimensional, the map ${}_1 x_3$ on $Q$ has a one-dimensional kernel $K$.
	\item $S \oplus K$ is then a submodule of dimension $2 \alpha_1+2\alpha_2+\alpha_3$, which violates stability. 
\end{itemize} 
The issue causing the counter-examples in this section seem somewhat distinct from those used in Theorems \ref{th:main} and \ref{th:cond2}.
\end{Example}

\section{Additional local conditions}
In the rank two case studied in \cite{T21}, we believe we listed all ``local" conditions, meaning our conditions catch all sub-words which by themselves imply stability fails. In the rank 3 cases studied here this is not the case, which implies our estimates should diverge exponentially from the actual multiplicities. Here are some examples demonstrating the issues, all in the case $s=2,t=1$.

\begin{Example} \label{ex:bad1}
	%Take $s=t=2$, and 
	A word of the form
	$$ ...112211{\color{red}22111}2...$$
	does not (depending on what comes before or after) immediately fail any of the conditions. But the segment colored red corresponds to the sub-quotient $Q$ with socle data $22111$. By the preprojective relation and dimension counts this implies a submodule with data 
	$$...1122{\color{red} 122}.$$
	Now the 1 in the sub-quotient $Q'$ indicated in red has no freedom to map to the $2$s on the level below, which contradicts this being the socle filtration. Thus no words of this form correspond to stable components (or even to unstable components). 
	Note that if the word was of the form 
	$$ ...112311{\color{red} 22111}2...$$
	instead there would have been no problem. 
\end{Example}

\begin{Example} \label{ex:bad2}
	One can of course find conditions to catch Example \ref{ex:bad1}. But one can just find worse examples. For instance, consider string data of the form
	$$.... 1^?2^71^72^71^72^71^72^71^8 2^?...$$
	Then the $2^71^8$ at the right implies a $1^6 2^7$ then a   $2^51^6$  and so on down to a $21^2$ subquotient, and then the two has no freedom to map to any 1s, causing a contradiction. 
	But
	$$.... 1^?2^731^72^71^72^71^72^71^8 2^?...$$
	would not have had this problem. So we need to consider many steps in the path to find the problem. 
\end{Example}

%Problems like that in Example \ref{ex:bad2} are a little hard to detect one must consider many steps of the filtration to find them. 

\begin{Example} \label{ex:strangeneg} Another issue arises because, in rank 3, there are elements of the positive root lattice of negative norm which are not imaginary roots. For instance, consider the word
	$$ 1^{30} 2^{20} 3^{10} 1^{15} 2^{20} 3^{6} 1^{33} 2^{40} 3^{20}.$$
	This satisfies all of our conditions, including the refined conditions in Theorem \ref{th:cond2} below. But the sub-quotient $2^{20} 3^{6} 1^{33}$ implies that this will not be stable. To see why, notice that the $2^{20} 3^{6}$ only has freedom to map to $13$ dimensions, so this implies the existence of a sub-quotient at the next step down with data $1^{13-e_1} 2^{20} 3^{6}$ for some $e_1\geq 0$. The $1^{13-e_1}$ then has the freedom to map to only $6-2e_1$ dimensions in degree 2 and $7-e_1$ in degree $3$, giving a sub-quotient with data
	$2^{6-2e_1-e_2} 3^{7-e_1-e_3} 1^{13-e_1}$ for some $e_2,e_3 \geq 0$. Then at the next step the $2^{6-2e_1-e_2} 3^{7-e_1-e_3}$ can map to a module in degree one of dimension at most $6-4e_1-2e_2-e_3$, which in particular is less then $7-e_1-e_3$. So, at this level the map ${}_1 x_3$ has a kernel, which contradicts the definition of the socle filtration. 
	
	The issue here is related to the fact that $ \beta= 33\alpha_1+20\alpha_2+6 \alpha_3$ has
	$$s_2s_3s_1s_2s_3s_1 \beta = 6\alpha_1+6\alpha_2-\alpha_3.$$
	The fact that we have not been able to rule this out using our ratio conditions is related to the fact that  $|\beta|^2=-110<0$ %. So $\beta$ looks like an imaginary root, 
	but $\beta$ is not an imaginary root since it reflects to a weight with negative coefficients. In this case, if one keeps applying $s_2s_3s_1$ more times the coefficients all become positive again. Our ratio condition is really about ensuring positive coefficients in a limit, so it misses the problem. 
%	Our ratio condition fails to capture a number of such situations. % instances where stability fails due to a sub-quotient having dimension $\beta$ where $$\cdots s_2s_3s_1s_2s_3s_1s_2s_3s_1 \beta $$ results in a negative coefficient.
%	Conditions like this could perhaps be added to our list of conditions in Theorem \ref{th:main}, but it gets complicated. 
	
\end{Example}

\section{Computational data }  %\label{s:st2}
In the cases $s=t=2$ and $
s=2,t=1$ we computed our estimates in many examples using Python\footnote{The code is available at \href{https://webpages.math.luc.edu/~ptingley/lecturenotes/rootmultscode-rank3.py}{https://webpages.math.luc.edu/$\sim$ptingley/}}. The actual multiplicities for $s=2,t=1$ can be found in \cite[Chapter 11]{Kac:1990}. The multiplicities in the case $s=t=2$ were calculated by Alex Feingold using mathematica code originally written by Stephen Miller implementing the Peterson algorithm to calculate the multiplicities, see \cite{Pete:rec}. The results are given in Figures \ref{fig:21} and \ref{fig:22}. 

Note that, for the case $s=t=2$, the symmetry of the Dynkin diagram implies that for any $a,k,\ell$, the roots $a\alpha_1+k \alpha_2+\ell \alpha_3$  and $a\alpha_1+\ell \alpha_2+k \alpha_3$  have the same multiplicity. But our method breaks this symmetry and does not always give the same estimates in these cases. See for example the data for the roots $9\alpha_1+8\alpha_2+7\alpha_3$ and $9\alpha_1+7\alpha_2+8\alpha_3$. The symmetry breaking only comes into effect when $a$ and $k+\ell$ are not relatively prime, so that Dyck paths that touch the diagonal are possible.% So, for instance, our estimates are the same for
%$9\alpha_1+9\alpha_2+7\alpha_3$ and %$9\alpha_1+7\alpha_2+9\alpha_3.$

In both tables, we only lists roots which are minimal in the sense that they cannot be reflected via the Weyl group to roots with smaller coefficients. As in Example \ref{ex:clean-nonminimal} some other cases are non-the-less interesting.

\begin{figure}
\begin{tabular}{|p{2cm}||p{3cm}|p{3cm}|p{3cm}|}
 \hline 
 Case: (A,B,C)  & Actual root multiplicity & Estimate with Theorems \ref{th:main} and \ref{th:cond2} & Estimate with just Theorem \ref{th:main}\\
 \hline
(1, 1, 1)   & 1      & 1      & 1      \\
(2, 1, 1)   & 2      & 2      & 2      \\
(2, 2, 1)   & 3      & 3      & 3      \\
(2, 1, 2)   & 3      & 3      & 3      \\
(3, 2, 1)   & 4      & 4      & 5      \\
(3, 1, 2)   & 4      & 4      & 5      \\
(3, 3, 1)   & 5      & 5      & 5      \\
(3, 1, 3)   & 5      & 5      & 5      \\
(4, 3, 1)   & 7      & 7      & 9      \\
(4, 1, 3)   & 7      & 7      & 11      \\
(4, 4, 1)   & 10     & 10     & 10     \\
(4, 1, 4)   & 10     & 10     & 10     \\
(3, 2, 2)   & 10     & 10     & 10     \\
(3, 3, 2)   & 12     & 12     & 12     \\
(3, 2, 3)   & 12     & 12     & 12     \\
(5, 4, 1)   & 13     & 13     & 19     \\
(5, 1, 4)   & 13     & 13     & 25     \\
(5, 5, 1)   & 16     & 16     & 18     \\
(5, 1, 5)   & 16     & 16     & 18     \\
(6, 5, 1)   & 21     & 22*     & 35     \\
(6, 1, 5)   & 21     & 21     & 35     \\
(6, 6, 1)   & 28     & 28     & 34     \\
(6, 1, 6)   & 28     & 28     & 34     \\
(4, 3, 2)   & 25     & 25     & 25     \\
(4, 2, 3)   & 25     & 25     & 25     \\
(7, 6, 1)   & 35     & 36*     & 68     \\
(7, 1, 6)   & 35     & 35     & 107     \\
(7, 7, 1)   & 43     & 43     & 61     \\
(7, 1, 7)   & 43     & 43     & 61     \\
(8, 7, 1)   & 55     & 58     & 124     \\
(8, 1, 7)   & 55     & 56*     & 212    \\
(4, 3, 3)   & 46     & 46     & 46     \\
(4, 4, 3)   & 58     & 58     & 58     \\
(4, 3, 4)   & 58     & 58     & 58     \\
(5, 4, 2)   & 61     & 61     & 64     \\
(5, 2, 4)   & 61     & 61     & 64     \\
(8, 8, 1)   & 70     & 70     & 114    \\
(8, 1, 8)   & 70     & 70     & 114    \\
%(5, 4, 2)   & 61     & 61     & 64     \\
%(4, 4, 3)   & 58     & 58     & 58     \\
%(4, 3, 4)   & 58     & 58     & 58     \\
%(5, 5, 2)   & 73     & 73     & 76     \\
%(9, 9, 1)   & 105    & 106    & 204    \\
%(10, 10, 1) & 161    & 165    & 377    \\
%(6, 5, 2)   & 133    & 133    & 147    \\
%(5, 4, 3)   & 153    & 153    & 156    \\
%(5, 3, 4)   & 153    & 153    & 156    \\
%(11, 11, 1) & 236    & 244    & 680    \\
%(5, 5, 3)   & 181    & 182    & 185    \\
%(5, 3, 5)   & 181    & 182    & 185    \\
%(7, 6, 2)   & 283    & 283    & 340    \\
%(5, 4, 4)   & 262    & 262    & 265    \\
%(7, 7, 2)   & 335    & 336    & 392    \\
%(5, 5, 4)   & 307    & 308    & 311    \\
%(5, 4, 5)   & 307    & 308    & 311    \\
%(6, 5, 3)   & 439    & 439    & 458    \\
%(8, 7, 2)   & 565    & 567    & 744    \\
%(9, 8, 2)   & 1100   & 1108   & 1612   \\
%(6, 5, 4)   & 969    & 970    & 990    \\
\hline
\hline
(10, 7, 6)  & 251656 & 251911 & 284878 \\
(9, 8, 7)   & 273917 & 275221 & 281488 \\
(9, 7, 8)   & 273917 & 275046 & 281363 \\
(9, 9, 7)   & 303947 & 306371 & 311847 \\
(9, 7, 9)   & 303947 & 306371 & 311847 \\
 \hline
\end{tabular}

\caption{\label{fig:22} Root multiplicities and our estimates for $s=t=2$. The table begins by systematically looking at small roots, where the three cases where are best estimate does not agree with the actual multiplicity are marked with a *. We then jumps to some of the largest roots we were able to work with.}
\end{figure}

\begin{figure}
\begin{tabular}{|p{2cm}||p{3cm}|p{3cm}|p{3cm}|}

\hline

Case: (A,B,C)  & Actual multiplicity & Bound using  Theorems \ref{th:main} and \ref{th:cond2}  & Bound using just Theorem \ref{th:main}\\

\hline

(2, 2, 1)   & 2    & 2    & 2    \\

(3, 3, 1)   & 3    & 3    & 3    \\

(4, 3, 2)   & 5    & 5    & 6    \\

(4, 4, 1)   & 5    & 5    & 6    \\

(5, 5, 1)   & 7    & 7    & 9    \\

(5, 4, 2)   & 11   & 11   & 15   \\

(6, 6, 1)   & 11   & 11   & 16   \\

(5, 5, 2)   & 15   & 15   & 18   \\

(7, 7, 1)   & 15   & 15   & 24   \\

(6, 5, 2)   & 22   & 22   & 34   \\

(8, 8, 1)   & 22   & 22   & 39   \\

(6, 5, 3)   & 30   & 33   & 46   \\

(9, 9, 1)   & 30   & 30   & 61   \\

(7, 6, 2)   & 42   & 42   & 72   \\

(10, 10, 1) & 42   & 42   & 96   \\

(7, 7, 2)   & 56   & 56   & 79   \\

(11, 11, 1) & 56   & 56   & 148  \\

(7, 6, 3)   & 77   & 83   & 121  \\

(8, 7, 2)   & 77   & 77   & 146  \\

(12, 12, 1) & 77   & 77   & 233  \\

(7, 7, 3)   & 101  & 101  & 134  \\

(9, 8, 2)   & 135  & 137  & 283  \\

(8, 7, 3)   & 176  & 187  & 296  \\

(9, 9, 2)   & 176  & 176  & 287  \\

(8, 7, 4)   & 231  & 253  & 379  \\

(8, 8, 3)   & 231  & 233  & 316  \\

(10, 9, 2)  & 231  & 235  & 531  \\

(9, 7, 4)   & 297  & 317  & 725  \\

(9, 8, 3)   & 385  & 410  & 682  \\

(11, 10, 2) & 385  & 399  & 974  \\

(9, 8, 3)   & 385  & 410  & 682  \\

(11, 11, 2) & 490  & 499  & 934  \\

(9, 8, 4)   & 627  & 674  & 1062 \\

(9, 9, 4)   & 792  & 807  & 1107 \\

(10, 9, 3)  & 792  & 839  & 1498 \\

(10, 8, 5)  & 1002 & 1218 & 2335 \\

(10, 9, 4)  & 1574 & 1656 & 2754 \\

(11, 10, 3) & 1574 & 1673 & 3161 \\

(10, 9, 5)  & 1957 & 2167 & 3404 \\

(11, 9, 4)  & 1957 & 2029 & 5113 \\

(11, 11, 3) & 1956 & 2000 & 3134 \\

(11, 9, 5)  & 3007 & 3492 & 6942 \\

(11, 10, 4) & 3713 & 3912 & 6776 \\

\hline

\end{tabular}

\caption{\label{fig:21} Root Multiplicity Data: s=2, t=1.}

\end{figure}


\begin{thebibliography}{11}

 \bibitem[BKT14]{BKT:2014}
Pierre Baumann, Joel Kamnitzer and Peter Tingley. Affine Mirkovi\'c-Vilonen polytopes. { Publ. Math. Inst. Hautes \'Etudes Sci.} {\bf 120} (2014) 113--205, 
\arxiv{1110.3661}.

\bibitem[BM79]{BM}
S. Berman and R. V. Moody. Lie algebra multiplicities. { Proc. Amer. Math. Soc.} {\bf 76}, no. 2, 1979, 223--228.

\bibitem[CFL14]{CFL:2014}
L.~Carbone, W.~Freyn, and K-H.~Lee.
Dimensions of imaginary root spaces of hyperbolic
Kac-Moody algebras.
{ Contemporary Mathematics}
{\bf 623} , 2014, 23--40
\arxiv{1305.3318}.

%\bibitem[Chan]{code}
%P.~Chan. Pythan code. {\color{red} post} 

\bibitem[FF83]{FF}
A. J. Feingold and I. B. Frenkel. A hyperbolic Kac-Moody algebra and the theory of Siegel modular forms of
genus 2, Math. Ann. {\bf 263} (1983), no. 1, 87–144.

\bibitem[Fre85]{Frenkel-conj}
I. B. Frenkel, Representations of Kac-Moody algebras and dual resonance models, Applications of group theory
in physics and mathematical physics (Chicago, 1982), Lectures in Appl. Math., {\bf vol. 21}, Amer. Math. Soc.,
Providence, RI, 1985, 325--353.


\bibitem[HK02]{HK:2002}
Jin Hong and Seok-Jin Kang.
\newblock {\em Introduction to quantum groups and crystal bases}, volume~42 of
  {\em Graduate Studies in Mathematics}.
\newblock American Mathematical Society, Providence, RI, 2002.


\bibitem[Kac90]{Kac:1990}
Victor~G. Kac.
\newblock {\em Infinite-dimensional {L}ie algebras}.
\newblock Cambridge University Press, Cambridge, third edition, 1990.

\bibitem[KMW87]{KMW}
V. G. Kac, R. V. Moody and M. Wakimoto, On E10, Differential geometrical methods in theoretical physics (Como, 1987), NATO Adv. Sci. Inst. Ser. C Math. Phys. Sci., vol. 250, Kluwer Acad. Publ., Dordrecht, 1988,
pp. 109--128.

\bibitem[KLL17]{KLL}
Seok-Jin Kang, Kyu-Hwan Lee and Kyungyong Lee. A combinatorial approach to root multiplicities of rank 2 hyperbolic Kac-Moody algebras.
Communications in Algebra.
{\bf 45} Issue 11 (2017) 4785--4800. \arxiv{1501.02026}.

\bibitem[KM95]{KM95}
S.-J. Kang and D. J. Melville, Rank 2 symmetric hyperbolic Kac-Moody algebras, Nagoya Math. J. {\bf 140} (1995),
41--75.


\bibitem[Kas95]{Kashiwara:1995}
M.~Kashiwara.
\newblock On crystal bases.
 { Canadian Math. Soc. Conf. Proc. }{\bf 16} 155--197, 1995.

\bibitem[KS97]{KS:1997}
Masaki Kashiwara and Yoshihisa Saito.
\newblock Geometric construction of crystal bases.
\newblock {\em Duke Math. J.}, {\bf 89}(1):9--36, 1997.
\newblock \arxiv{q-alg/9606009}.

 \bibitem[Lit98]{Littelmann:1998}
P. Littelmann. Cones, crystals, and patterns. { Transform. Groups} {\bf 3} (1998), no. 2, 145--179.


\bibitem[NT18]{NT}
Vinoth Nandakumar and Peter Tingley. Quiver varieties and crystals in symmetrizable type via modulated graphs. Math. Res. Lett. {\bf 25} (2018), no. 1, 159--180. \arxiv{1606.01876} 


\bibitem[Pet83]{Pete:rec}
 D. H. Peterson. Freudenthal-type formulas for root and weight multiplicities, preprint
(unpublished) (1983).

\bibitem[Rud97]{Rudakov97}
A.~Rudakov, \textit{Stability for an abelian category,} J.\ Algebra
\textbf{197} (1997), 231--245.

\bibitem[Tin21]{T21} Peter Tingley. 
A quiver variety approach to root multiplicities. Algebraic Combinatorics, Volume 4 (2021) no. 1, pp. 163--174.
 \arxiv{1910.04637}. 


\bibitem[TW16]{TW}
Peter Tingley and Ben Webster
Mirkovi\'c-Vilonen polytopes and Khovanov-Lauda-Rouquier algebras.  Compositio Math. {\bf 152} (2016) 1648--1696.  \arxiv{1210.6921}


\end{thebibliography}
\end{document}